\documentclass{amsart}

\usepackage{amssymb}

\newcommand{\norm}[1]{\lVert#1\rVert}

\newtheorem{theorem}{Theorem}[section]
\newtheorem{lemma}[theorem]{Lemma}

\newtheorem{proposition}[theorem]{Proposition}

\theoremstyle{definition}
\newtheorem{definition}[theorem]{Definition}
\newtheorem{example}[theorem]{Example}

\theoremstyle{remark}

\numberwithin{equation}{section}

\begin{document}

\title{Constructing Banach ideals using upper $\ell_p$-estimates}

\author{Ben Wallis}
\address{Department of Mathematical Sciences, Northern Illinois University, DeKalb, IL 60115}
%\curraddr{}
\email{wallis@math.niu.edu}
%\thanks{This research was conducted as part of the author's Ph.D. program, under advisor Gleb Sirotkin.  The author sincerely thanks his advisor for all his guidance and patience.}

\begin{abstract}
Using upper $\ell_p$-estimates for normalized weakly null sequence images, we describe a new family of operator ideals $\mathcal{WD}_{\ell_p}^{(\infty,\xi)}$ with parameters $1\leq p\leq\infty$ and $1\leq\xi\leq\omega_1$.  These classes contain the completely continuous operators, and are distinct for all choices $1\leq p\leq\infty$ and, when $p\neq 1$, for all choices $\xi\neq\omega_1$.  For the case $\xi=1$, there exists an ideal norm $\norm{\cdot}_{(p,1)}$ on the class $\mathcal{WD}_{\ell_p}^{(\infty,1)}$ under which it forms a Banach ideal.
\end{abstract}

\maketitle

\section{Introduction}

The roots of the theory of operator ideals extend at least as far back as 1941 when J.W. Calkin observed that if $\mathcal{H}$ is a Hilbert space, then the subspaces of finite-rank operators, compact operators, and Hilbert-Schmidt operators all form multiplicative ideals in the space $\mathcal{L}(\mathcal{H})$ of continuous linear operators on $\mathcal{H}$ (\cite{Ca41}).  However, the concept of an ideal as a class of operators between arbitrary Banach spaces developed more recently, with the first thorough treatment of the subject, a monograph by Albrecht Pietsch, appearing in 1978 (\cite{Pi78}).  %There Pietsch cataloged and studied a wide variety of operator ideals. %, including but not limited to the finite-rank, approximable, compact, weakly compact, completely continuous, absolutely summing, separable, strictly singular, strictly cosingular, nuclear, and integral operators.

In this paper we define and study a new family of operator ideals $\mathcal{WD}_{\ell_p}^{(\infty,\xi)}$ with parameters $1\leq p\leq\infty$ and $1\leq\xi\leq\omega_1$, where $\omega_1$ denotes the first uncountable ordinal.  For any fixed value of $\xi$, these ideals are distinct for all choices of $p$, which is to say that for any $1\leq p<q\leq\infty$ there exist Banach spaces $X$ and $Y$ for which the components satisfy $\mathcal{WD}_{\ell_p}^{(\xi,\infty)}(X,Y)\neq\mathcal{WD}_{\ell_q}^{(\xi,\infty)}(X,Y)$.  It remains an open question whether, for fixed $1<p\leq\infty$, the ideals are distinct for all choices of $\xi$.  However, we do obtain a partial positive answer by establishing $\mathcal{WD}_{\ell_p}^{(\infty,\xi)}\neq\mathcal{WD}_{\ell_p}^{(\infty,\omega_1)}$ whenever $\xi\neq\omega_1$ and $p\neq 1$.  We shall also see that $\mathcal{WD}_{\ell_p}^{(\infty,\xi)}$ always strictly includes the ideal of completely continuous operators $\mathcal{V}$, which shows that they are distinct from some other notable families of operator ideals with a parameter related to the $\ell_p$ spaces.  For instance, let $\mathcal{N}_p$, $\mathcal{I}_p$, and $\Pi_p$ denote the ideals of $p$-nuclear, $p$-integral, and absolutely $p$-summing operators, respectively.  Then $\mathcal{N}_p\subsetneq\mathcal{I}_p\subsetneq\Pi_p\subsetneq\mathcal{V}\subsetneq\mathcal{WD}_{\ell_p}^{(\infty,\xi)}$ (cf., e.g., Proposition 22 in \cite{Wo91} together with Theorem 2.17 in \cite{DJT95}).

Of special interest are the those operator ideals whose components are always norm-closed.  For instance, given arbitrary Banach spaces $X$ and $Y$, the compact, weakly compact, and completely continuous operators from $X$ into $Y$ are always norm-closed in $\mathcal{L}(X,Y)$, whereas the finite-rank operators are not.  We shall see that, when $p\neq 1$, there always exist separable spaces $X$ for which $\mathcal{WD}_{\ell_p}^{(\infty,\xi)}(X)$ fails to be norm-closed in $\mathcal{L}(X)$, and when $p\neq\infty$, we can choose $X$ to be reflexive.  Nevertheless, in the case $\xi=1$, we can construct an ideal norm $\norm{\cdot}_{(p,1)}$ for the class $\mathcal{WD}_{\ell_p}^{(\infty,1)}$ so that it forms a Banach ideal, that is, a ``nice'' norm assignment for each component space $\mathcal{WD}_{\ell_p}^{(\infty,1)}(X,Y)$ under which it becomes a Banach space.

The ideas for the construction of this family originate with \cite{BF11} and \cite{ADST09}.  In \cite{BF11}, the authors defined the subset $\mathcal{WS}(X,Y)$ of {\bf $\boldsymbol{(w_n)}$-singular} operators in $\mathcal{L}(X,Y)$ as those operators $T$ for which, given any normalized basic sequence $(x_n)$ in $X$, the image sequence $(Tx_n)$ fails to dominate $(w_n)$.  Here, $(w_n)$ is taken to be some normalized 1-spreading basis for some fixed Banach space $W$.  They showed that when $(w_n)$ is the summing basis for $c_0$, the unit vector basis for $\ell_1$, or the unit vector basis for $c_0$, the resulting classes $\mathcal{WS}$ are the norm-closed ideals of weakly compact, Rosenthal, or compact operators, respectively.  Meanwhile, in \cite{ADST09} the authors constructed and studied classes of operators based on Schreier family support.  In particular, they defined $\mathcal{SS}_\xi$, the {\bf $\boldsymbol{\mathcal{S}_\xi}$-strictly singular} operators, as the class of all continuous linear Banach space operators $T$ for which if $(x_n)$ is any normalized basic sequence in the domain space, for any $\epsilon>0$ there exists some $z\in[x_n]$ with support lying in the $\xi$th Schreier family $\mathcal{S}_\xi$, and satisfying $\norm{Tz}<\epsilon\norm{z}$.

In this paper, we use similar ideas to produce operator ideals with certain nice properties.  However, whereas classes $\mathcal{WS}$ and $\mathcal{SS}_\xi$ were constructed using normalized basic sequences and singular estimates on their images, for the classes $\mathcal{WD}_{\ell_p}^{(\infty,\xi)}$ we instead use normalized weakly null sequences and uniform upper estimates.  Since continuous linear operators preserve weak convergence, the choice of weakly null sequences in place of basic sequences allows us to show that the classes $\mathcal{WD}_{\ell_p}^{(\infty,\xi)}$ are indeed multiplicative ideals between arbitrary Banach spaces.  The choice of uniform upper estimates instead of singular estimates then gives us a natural way to show that each class $\mathcal{WD}_{\ell_p}^{(\infty,\xi)}$ is closed under addition.

Now we shall take a moment to recall some essential definitions and basic facts relevant to our project.  Let $\mathcal{J}$ be a subclass of the class $\mathcal{L}$ of all continuous linear operators between Banach spaces, such that for each pair of Banach spaces $X$ and $Y$, $\mathcal{J}(X,Y):=\mathcal{L}(X,Y)\cap\mathcal{J}$ is a linear subspace containing all the finite-rank operators from $X$ into $Y$.  We call $\mathcal{J}$ an \textbf{operator ideal} if whevenever $W,X,Y,Z$ are Banach spaces and $T\in\mathcal{J}(X,Y)$, then for all operators $A\in\mathcal{L}(W,X)$ and $B\in\mathcal{L}(Y,Z)$ we have $BTA\in\mathcal{J}(W,Z)$.  An {\bf ideal norm} with respect to an operator ideal $\mathcal{J}$ is a rule $\rho$ that assigns to every $T\in\mathcal{J}(X,Y)$, a nonnegative real value $\rho(T)$, and satisfying the following conditions for all Banach spaces $W$, $X$, $Y$, and $Z$.  First, $\rho(x^*\otimes y)=\norm{x^*}\norm{y}$ for all $x^*\in X^*$ and $y\in Y$, where $x^*\otimes y$ is viewed as the 1-dimensional operator $x\mapsto x^*(x)y$ lying in $\mathcal{J}(X,Y)$; second, $\rho(S+T)\leq\rho(S)+\rho(T)$ for all $S,T\in\mathcal{J}(X,Y)$; and third, $\rho(BTA)\leq\norm{B}\rho(T)\norm{A}$ for all $T\in\mathcal{J}(X,Y)$, $A\in\mathcal{L}(W,X)$, and $B\in\mathcal{L}(Y,Z)$.  A {\bf Banach ideal} is then an operator ideal $\mathcal{J}$ equipped with an an ideal norm $\rho$ such that all components $\mathcal{J}(X,Y)$ are complete with respect to the norm on that space induced by $\rho$.

We will also need to use the \textbf{Schreier families}.  These are denoted $\mathcal{S}_\xi$ for each countable ordinal $0\leq\xi<\omega_1$, and we must define them as follows.  Put $\mathcal{S}_0:=\{\{n\}:n\in\mathbb{N}\}\cup\{\emptyset\}$ and $\mathcal{S}_1:=\{F\subset\mathbb{N}:\# F\leq\min F\}\cup\{\emptyset\}$.  Now fix a countable ordinal $1\leq\xi<\omega_1$.  In case $\xi=\zeta+1$ for some countable ordinal $1\leq\zeta<\omega_1$ we define $\mathcal{S}_\xi$ as the set containing $\emptyset$ together with all $F\subset\mathbb{N}$ such that there exist $n\in\mathbb{N}$ and $F_1<\cdots<F_n\in\mathcal{S}_\zeta$ satisfying $\{\min F_k\}_{k=1}^n\in\mathcal{S}_1$ and $F=\bigcup_{k=1}^nF_k$. In case $\xi$ is a limit ordinal we fix a strictly increasing sequence $(\zeta_n)$ of non-limit-ordinals satisfying $\sup_n\zeta_n=\xi$, and define $\mathcal{S}_\xi:=\bigcup_{n=1}^\infty\{F\in\mathcal{S}_{\zeta_n}:n\leq F\}$.

Usually in the literature, the family of finite subsets of natural numbers is denoted $[\mathbb{N}]^{<\infty}$, or $\mathcal{P}_{<\infty}(\mathbb{N})$.  However, for convenience, let us abuse our notation and write this family as if it were the ``$\omega_1$th Schreier family.''  In other words, we set $\mathcal{S}_{\omega_1}:=\{F\subset\mathbb{N}:\#F<\infty\}$.  This will greatly simplify the writing.

The sets $\mathcal{S}_\xi$ ($1\leq\xi\leq\omega_1$) have some very nice properties, most especially that each is {\bf spreading}.  This means that if $\{m_1<\cdots<m_k\}\in\mathcal{S}_\xi$ and $\{n_1<\cdots<n_k\}$ satisfies $m_i\leq n_i$ for all $1\leq i\leq k$, then $\{n_1<\cdots<n_k\}\in\mathcal{S}_\xi$ also holds.  They are also {\bf hereditary}, which means that if $E\in\mathcal{S}_\xi$ and $F\subseteq E$ then $F\in\mathcal{S}_\xi$.  Contrary to what we might expect, though, the Schreier families are {\it not} increasing under the inclusion relation.  However, it is easily seen that, for all $1\leq\xi\leq\omega_1$, we have $\mathcal{S}_1\subseteq\mathcal{S}_\xi$, and in particular we have $\{k\}\in\mathcal{S}_\xi$ for all $k\in\mathbb{N}$.  Moreover, the Schreier families do behave somewhat nicely under the inclusion relation in the sense that, if $1\leq\zeta<\xi\leq\omega_1$ are ordinals, then there is $d=d(\zeta,\xi)\in\mathbb{N}$ such that for all $F\in\mathcal{S}_\zeta$ satisfying $d\leq\min F$ we have $F\in\mathcal{S}_\xi$.

%Let's also review the following standard facts.

%\begin{theorem}[Kadets-Pe\l czy\'{n}ski, Theorem 1.5.6 in \cite{AK06}]  Let $S$ be a bounded subset of a Banach space $X$ whose closure excludes zero.  Then $S$ fails to contain a basic sequence if and only if its weak closure is weakly compact and also fails to contain zero.\end{theorem}

%\begin{corollary}\label{reflexive-weakly-null}Every seminormalized basic sequence in a reflexive Banach space is weakly null.\end{corollary}

%\begin{proof}Let $(x_n)$ be seminormalized basic, and suppose it is not weakly null.  Then there is $x^*\in X^*$ such that $x^*(x_{n_k})\to r$ for some $r\neq 0$.  Since $X$ is reflexive, we can pass to a further subsequence if necessary so that $x_{n_k}\stackrel{w}{\to}x$ for some $x\in X$.  By Eberline-Smulian, this means $(x_{n_k})$ only has the one weak closure point $x$.  By Kadets-Pelczynski this closure point must be zero, which means $x^*(0)=x^*(x)=r\neq 0$, which is impossible.\end{proof}

We will appeal several times to the Bessaga-Pe\l czy\'{n}ski Selection Principle.  However, the version that we need is slightly stronger than typically stated in the literature.  More specifically, we need a small uniform bound on the equivalence constant.  The proof is practically identical to the standard small perturbations and gliding hump arguments found, for instance, in Theorem 1.3.9 and Proposition 1.3.10 from \cite{AK06}.

\begin{theorem}[Uniform Bessaga-Pe\l czy\'{n}ski Selection Principle]  Suppose $X$ is a Banach space with a basis $(e_i)$, and corresponding coefficient functionals $(e_i^*)\subset X^*$.  Let $(x_n)\subset X$ be a sequence satisfying $\lim_{n\to\infty}\norm{x_n}=1$ and $\lim_{n\to\infty}e_i^*(x_n)=0$ for all $i\in\mathbb{N}$.  Then for any $\epsilon>0$, there exists a subsequence $(x_{n_k})$ which is $(1+\epsilon)$-congruent to a normalized block basis of $(e_i)$.  %(This means there exists an isomorphism $U:X\to X$ such that $(Ux_{n_k})$ is a normalized block basis of $(e_i)$ and $\norm{U^{-1}}\norm{U}\leq 1+\epsilon$.)
\end{theorem}

We divide the remainder of this paper into sections 2 and 3.  In section 2 we define the classes $\mathcal{WD}_{\ell_p}^{(\infty,\xi)}$, and establish that, for the nontrivial case $p\neq 1$, they fail to be norm-closed, but as long as $\xi=1$ they form Banach ideals.  Then, in section 3 we discuss the significance of the parameters $p$ and $\xi$.

\section{The operator ideals $\mathcal{WD}_{\ell_p}^{(\infty,\xi)}$}

%In this section, we will lay out the building blocks, so to speak, of the ideals $\mathcal{WD}_{\ell_p}^{(\infty,\xi)}$, and show that one of them coincides with the completely continuous operators.  We therefore require a few more definitions.

%\begin{definition}Let $1\leq\xi<\omega_1$ be a countable ordinal.  Then we define $\mathcal{A}_\xi:=\{(\alpha_k)\in c_{00}:\text{supp}(\alpha_k)\in\mathcal{S}_\xi\}$, the set of all scalar sequences with support in the $\xi$th Schreier family.\end{definition}

%\noindent Notice that since $\mathcal{S}_\xi$ is always hereditary, we could instead write

%\begin{equation*}\mathcal{A}_\xi=\{(\alpha_k)\in c_{00}:\text{supp}(\alpha_k)\subseteq E\text{ for some }E\in\mathcal{S}_\xi\}.\end{equation*}

Let us state formally the definition of classes $\mathcal{WD}_{\ell_p}^{(\infty,\xi)}$.

\begin{definition}Let $X$ and $Y$ be Banach spaces, and fix some constants $0\leq C<\infty$ and $1\leq p\leq\infty$, and some countable ordinal $1\leq\xi<\omega_1$.  Put $\mathcal{A}_\xi:=\{(\alpha_k)\in c_{00}:\text{supp}(\alpha_k)\in\mathcal{S}_\xi\}$, the set of all scalar sequences with support in the $\xi$th Schreier family.  Then we denote by $\mathcal{WD}_{\ell_p}^{(C,\xi)}(X,Y)$ the set of all operators $T\in\mathcal{L}(X,Y)$ for which, given $\epsilon>0$, each normalized weakly null sequence $(x_n)\subset X$ admits a subsequence $(x_{n_k})$ such that for all $(\alpha_k)\in\mathcal{A}_\xi$, the estimate $\norm{\sum\alpha_kTx_{n_k}}\leq(C+\epsilon)\norm{(\alpha_k)}_{\ell_p}$ holds.  Then we set $\mathcal{WD}_{\ell_p}^{(\infty,\xi)}(X,Y):=\bigcup_{C\geq 0}\mathcal{WD}_{\ell_p}^{(C,\xi)}(X,Y)$.\end{definition}

Immediate from the definitions and the inequality $\norm{(\alpha_k)}_{\ell_p}\leq\norm{(\alpha_k)}_{\ell_q}$ for all $(\alpha_k)\in c_{00}$ and $1\leq q\leq p\leq\infty$, we get the following relations.

\begin{proposition}\label{inclusion}Let $X$ and $Y$ be Banach spaces, and fix some constants $0\leq C\leq D\leq\infty$ and $1\leq q\leq p\leq\infty$, and some ordinal $1\leq\xi\leq\omega_1$.  Then $\mathcal{WD}_{\ell_p}^{(C,\xi)}(X,Y)\subseteq\mathcal{WD}_{\ell_q}^{(C,\xi)}(X,Y)$ and $\mathcal{WD}_{\ell_p}^{(C,\xi)}(X,Y)\subseteq\mathcal{WD}_{\ell_p}^{(D,\xi)}(X,Y)$.\end{proposition}

When checking that an operator satisfies the definition of $\mathcal{WD}_{\ell_p}^{(C,\xi)}$, the following Propositions will come in handy.

\begin{proposition}\label{seminormalized-tends-to-1}Let $X$ and $Y$ be Banach spaces, and fix constants $0\leq C<\infty$ and $1\leq p\leq\infty$, and some ordinal $1\leq\xi\leq\omega_1$.  Then $T\in\mathcal{WD}_{\ell_p}^{(C,\xi)}(X,Y)$ if and only if for all $\epsilon>0$ and every seminormalized weakly null sequence $(x_n)\subset X$ which admits a subsequence $(x_{n_k})$ satisfying $\norm{x_{n_k}}\to 1$, there exists a further subsequence $(x'_{n_k})$ such that for all $(\alpha_k)\in\mathcal{A}_\xi$, the estimate $\norm{\sum\alpha_kTx'_{n_k}}\leq(C+\epsilon)\norm{(\alpha_k)}_{\ell_p}$ holds.\end{proposition}

\begin{proof}We need only prove the ``only if'' part since the ``if'' part is obvious.  Suppose $T\in\mathcal{WD}_{\ell_p}^{(C,\xi)}(X,Y)$.  Let $(x_n)$ be a seminormalized weakly null sequence with a subsequence tending to 1 in norm, and pick $\epsilon>0$.  Let $1<\delta<1+\frac{\epsilon}{2C}$, which gives us $C\delta+\frac{\epsilon}{2}<C+\epsilon$, and pass to a further subsequence so that $\norm{x_{n_k}}\leq\delta$ for all $k$.  By definition of $T\in\mathcal{WD}_{\ell_p}^{(C,\xi)}(X,Y)$ we can pass to yet a further subsequence so that $(\frac{x_{n_k}}{\norm{x_{n_k}}})$ satisfies $\norm{\sum\alpha_kT\frac{x_{n_k}}{\norm{x_{n_k}}}}\leq(C+\frac{\epsilon}{2\delta})\norm{(\alpha_k)}_{\ell_p}$ for all $(\alpha_k)\in\mathcal{A}_\xi$.  Since also $(\norm{x_{n_k}}\alpha_k)\in\mathcal{A}_\xi$ for each $(\alpha_k)\in\mathcal{A}_\xi$, we get

\begin{multline*}\norm{\sum\alpha_kTx_{n_k}}=\norm{\sum\norm{x_{n_k}}\alpha_kT\frac{x_{n_k}}{\norm{x_{n_k}}}}\leq\left(C+\frac{\epsilon}{2\delta}\right)\norm{(\norm{x_{n_k}}\alpha_k)}_{\ell_p}\\\\\leq\left(C\delta+\frac{\epsilon}{2}\right)\norm{(\alpha_k)}_{\ell_p}\leq(C+\epsilon)\norm{(\alpha_k)}_{\ell_p}.\end{multline*}\end{proof}

%When checking to see whether $T\in\mathcal{WD}^{(C,\xi)}_{\ell_p}(X,Y)$, it is sufficient to check normalized weakly null (resp. normalized basic) sequences $(x_n)$ for which $(Tx_n)$ has no norm-null subsequence.  This fact we formalize as follows.

\begin{proposition}\label{norm-null}Let $X$ and $Y$ be Banach spaces, and fix constants $0\leq C<\infty$ and $1\leq p\leq\infty$.  If $(x_n)\subset X$ is a sequence for which $(Tx_n)$ has a norm-null subsequence, then given $\epsilon>0$, there exists a further subsequence $(x_{n_k})$ for which the estimate $\norm{\sum\alpha_kTx_{n_k}}\leq(C+\epsilon)\norm{(\alpha_k)}_{\ell_p}$ holds for all $(\alpha_k)\in c_{00}$.\end{proposition}

\begin{proof}Pick a subsequence so that $\norm{Tx_{n_k}}\leq\epsilon 2^{-k}$ and hence, by H\"older, if $q$ is conjugate to $p$ so that $\frac{1}{p}+\frac{1}{q}=1$, $\norm{\sum\alpha_kTx_{n_k}}\leq\epsilon\norm{(2^{-k}\alpha_k)}_{\ell_1}\leq\epsilon\norm{(2^{-k})}_{\ell_q}\norm{(\alpha_k)}_{\ell_p}\leq(C+\epsilon)\norm{(\alpha_k)}_{\ell_p}$ for any sequence $(\alpha_k)\in c_{00}$.\end{proof}

Recall that a linear operator between Banach spaces $X$ and $Y$ is called {\bf completely continuous} just in case it always sends weakly null sequences into norm-null ones.  We write $\mathcal{V}(X,Y)$ for the space of these completely continuous operators.  (As mentioned previously, $\mathcal{V}$ is a norm-closed operator ideal.)  Thus, Proposition \ref{norm-null} yields the following.

\begin{proposition}\label{VinWD}Let $X$ and $Y$ be Banach spaces, and let $1\leq p\leq\infty$, $0\leq C\leq\infty$, and $1\leq\xi<\omega_1$.  Then $\mathcal{V}(X,Y)\subseteq\mathcal{WD}_{\ell_p}^{(C,\xi)}(X,Y)$.\end{proposition}

%\begin{proof}Recall from Proposition \ref{norm-null} that whenever $(Tx_n)$ is norm-null, we get the desired inequality for all $(\alpha_k)\in c_{00}$ in the definition of $\mathcal{WD}_{\ell_p}^{(C,\xi)}(X,Y)$.  So if $T\in\mathcal{V}(X,Y)$ and $(x_n)$ is weakly null, we get $(Tx_n)$ is norm-null and hence $T\in\mathcal{WD}_{\ell_p}^{(C,\xi)}(X,Y)$.

%Next, let $T\in\mathcal{L}(X,Y)$ be any continuous linear operator.  If $X$ has the Schur property then $(x_n)$ and hence $(Tx_n)$ are both norm-null.  Next, recall that $T$ is norm-to-norm continuous if and only if it is weak-to-weak continuous.  Thus, if $(x_n)$ is weakly null then so is $(Tx_n)$.  So, if $Y$ has the Schur property then $(Tx_n)$ is still norm-null.  In either case, $T\in\mathcal{WD}_{\ell_p}^{(C,\xi)}(X,Y)$.\end{proof}

%In this section, we show that the classes $\mathcal{WD}_{\ell_p}^{(\infty,\xi)}$ are indeed operator ideals, but that they are not always norm-closed for any given pair of parameters $1<p\leq\infty$ and $1\leq\xi\leq\omega_1$.  They are, however, Banach ideals under a new norm, which we shall define shortly.

Let us observe, via several steps, that the class $\mathcal{WD}_{\ell_p}^{(\infty,\xi)}$ is indeed an operator ideal.

\begin{proposition}\label{multideal}Let $W$, $X$, $Y$, and $Z$ be Banach spaces, and fix constants $1\leq p\leq\infty$ and $0\leq C<\infty$, and some ordinal $1\leq\xi\leq\omega_1$.  Suppose $T\in\mathcal{WD}_{\ell_p}^{(C,\xi)}(X,Y)$ with $A\in\mathcal{L}(W,X)$ and $B\in\mathcal{L}(Y,Z)$.  Then $TA\in\mathcal{WD}_{\ell_p}^{(C\norm{A},\xi)}(W,Y)$ and $BT\in\mathcal{WD}_{\ell_p}^{(C\norm{B},\xi)}(X,Z)$.\end{proposition}

\begin{proof}Let's first show that $TA\in\mathcal{WD}_{\ell_p}^{(C\norm{A},\xi)}(W,Y)$.  Recall that an operator is weak-to-weak continuous if and only if it is norm-to-norm continuous.  Thus if $(w_n)$ is a normalized weakly null sequence in $W$, we get that $(Aw_n)$ is weakly null in $X$.  If it contains a norm-null subsequence then so does $TAw_n$, and so by Proposition \ref{norm-null} we are done.  Otherwise, we can pass to a subsequence if necessary so that $\norm{Aw_n}\to\delta$ for some $0<\delta\leq\norm{A}$.  Hence $\norm{\delta^{-1}Aw_n}\to 1$, and by Proposition \ref{seminormalized-tends-to-1} we get, for any $\epsilon>0$, a subsequence $(n_k)$ satisfying $\norm{\sum\alpha_kT\delta^{-1}Aw_{n_k}}\leq(C+\frac{\epsilon}{\delta})\norm{(\alpha_k)}_{\ell_p}$ and hence $\norm{\sum\alpha_kTAw_{n_k}}\leq(C\delta+\epsilon)\norm{(\alpha_k)}_{\ell_p}\leq(C\norm{A}+\epsilon)\norm{(\alpha_k)}_{\ell_p}$ for all $(\alpha_k)\in\mathcal{A}_\xi$.

Next, we show that $BT\in\mathcal{WD}_{\ell_p}^{(C\norm{B},\xi)}(X,Z)$.  Fix a normalized weakly null sequence $(x_n)\subset X$, and let $\epsilon>0$.  To make things nontrivial, we may assume $B\neq 0$.  Then we can find a subsequence $(x_{n_k})$ such that for all $(\alpha_k)\in\mathcal{A}_\xi$ we get $\norm{\sum\alpha_kTx_{n_k}}\leq(C+\frac{\epsilon}{\norm{B}})\norm{(\alpha_k)}_{\ell_p}$ and hence $\norm{\sum\alpha_kBTx_{n_k}}\leq\norm{B}\norm{\sum\alpha_kTx_{n_k}}\leq(C\norm{B}+\epsilon)\norm{(\alpha_k)}_{\ell_p}$.\end{proof}

%\begin{proposition}\label{leftmult}Let $X$, $Y$, and $Z$ be Banach spaces, and fix constants $1\leq p\leq\infty$ and $0\leq C<\infty$, and some ordinal $1\leq\xi\leq\omega_1$.  Suppose $T\in\mathcal{WD}_{\ell_p}^{(C,\xi)}(X,Y)$ (resp. $T\in\mathcal{BD}_{\ell_p}^{(C,\xi)}(X,Y)$) and $B\in\mathcal{L}(Y,Z)$.  Then $BT\in\mathcal{WD}_{\ell_p}^{(C\norm{B},\xi)}(W,Y)$ (resp. $BT\in\mathcal{BD}_{\ell_p}^{(C\norm{B},\xi)}(W,Y)$).\end{proposition}

%\begin{proof}Fix a normalized weakly null (resp. normalized basic) sequence $(x_n)\subset X$, and let $\epsilon>0$.  To make things nontrivial, we may assume $B\neq 0$.  Then we can find a subsequence $(x_{n_k})$ such that for all $(\alpha_k)\in\mathcal{A}_\xi$ we get

%\begin{equation*}\norm{\sum\alpha_kTx_{n_k}}\leq\left(C+\frac{\epsilon}{\norm{B}}\right)\norm{(\alpha_k)}_{\ell_p}\end{equation*}

%\noindent and hence

%\begin{equation*}\norm{\sum\alpha_kBTx_{n_k}}\leq\norm{B}\norm{\sum\alpha_kTx_{n_k}}\leq(C\norm{B}+\epsilon)\norm{(\alpha_k)}_{\ell_p}.\end{equation*}\end{proof}

By ``pushing out'' a scalar sequence $(\alpha_k)\in\mathcal{A}_\xi$, and using the spreading property of $\mathcal{S}_\xi$, we obtain the following obvious Lemma.

\begin{lemma}\label{subsequence-hereditary}Let $Y$ be a Banach space, and fix an ordinal $1\leq\xi\leq\omega_1$.  Suppose $(y_n)$ and $(y'_k)$ are sequences in $Y$ such that $(y'_k)_{k\geq m}$ is a subsequence of $(y_n)_{n\geq m}$ for some $m\in\mathbb{N}$.  If $(\alpha_k)\in\mathcal{A}_\xi$ satisfies $\min\text{supp}(\alpha_k)\geq m$ then there exists $(\beta_n)\in\mathcal{A}_\xi$ which satisfies $\sum\alpha_ky'_k=\sum\beta_ny_n$ and $\norm{(\alpha_k)}_{\ell_p}=\norm{(\beta_n)}_{\ell_p}$ for all $1\leq p\leq\infty$.
% which together satisfy $\norm{\sum\beta_nTx_n}\leq K\norm{(\beta_n)}_{\ell_p}$ for all $(\beta_n)\in\mathcal{A}_\xi$.  Then, for any subsequence $(n_k)$, we have $\norm{\sum\alpha_kTx_{n_k}}\leq K\norm{(\alpha_k)}_{\ell_p}$ for all $(\alpha_k)\in\mathcal{A}_\xi$.
\end{lemma}

%\begin{proof}Notice that there exists a subsequence $(y_{n_k})=(y_1,\cdots,y_{m-1},y'_m,y'_{m+1},\cdots)$ of $(y_n)$ for which $y'_k=y_{n_k}$ whenever $k\geq m$.  Define $(\beta_n)$ as follows.  Let $\beta_{n_k}:=\alpha_k$ for all $k\in\mathbb{N}$.  Then fill in the remaining values $\beta_n:=0$ when $n\neq n_k$ for any $k\in\mathbb{N}$.  This gives us $\sum\alpha_ky'_k=\sum\alpha_ky_{n_k}=\sum\beta_ny_n$ and $\norm{(\beta_n)}_{\ell_p}=\norm{(\alpha_k)}_{\ell_p}$.  Furthermore, by the spreading property of $\mathcal{S}_\xi$ we have $(\beta_n)\in\mathcal{A}_\xi$.
%Thus $\norm{\sum\alpha_kTx_{n_k}}=\norm{\sum\beta_nTx_n}\leq K\norm{(\beta_n)}_{\ell_p}=K\norm{(\alpha_k)}_{\ell_p}$.\end{proof}

\begin{proposition}\label{addition}Let $X$ and $Y$ be Banach spaces, and fix constants $0\leq C,D<\infty$ and $1\leq p\leq\infty$.  Then for any $S\in\mathcal{WD}_{\ell_p}^{(C,\xi)}(X,Y)$ and $T\in\mathcal{WD}_{\ell_p}^{(D,\xi)}(X,Y)$ we have $S+T\in\mathcal{WD}_{\ell_p}^{(C+D,\xi)}(X,Y)$.\end{proposition}

\begin{proof}Let $\epsilon>0$ and pick a normalized weakly null sequence $(x_n)\subset X$.  By definition of $S\in\mathcal{WD}_{\ell_p}^{(C,\xi)}(X,Y)$ applied to $\frac{\epsilon}{2}>0$ and $(x_n)$ we get a subsequence $(x_{n_k})$ such that for all $(\alpha_k)\in\mathcal{A}_\xi$, the estimate $\norm{\sum\alpha_kSx_{n_k}}\leq(C+\frac{\epsilon}{2})\norm{(\alpha_k)}_{\ell_p}$ holds.  Next, apply the definition of $T\in\mathcal{WD}_{\ell_p}^{(D,\xi)}(X,Y)$ to $\frac{\epsilon}{2}>0$ and $(x_{n_k})$ to find a further subsequence $(k_i)$ such that for all $(\alpha_i)\in\mathcal{A}_\xi$ we get $\norm{\sum\alpha_iTx_{n_{k_i}}}\leq(D+\frac{\epsilon}{2})\norm{(\alpha_i)}_{\ell_p}$.  Notice that since $(x_{n_{k_i}})$ is a subsequence of $(x_{n_k})$, then for each scalar sequence $(\alpha_i)\in\mathcal{A}_\xi$, by Lemma \ref{subsequence-hereditary}, $\norm{\sum\alpha_i(S+T)x_{n_{k_i}}}\leq\norm{\sum\alpha_iSx_{n_{k_i}}}+\norm{\sum\alpha_iTx_{n_{k_i}}}\leq(C+D+\epsilon)\norm{(\alpha_i)}_{\ell_p}$.\end{proof}

From Propositions \ref{VinWD}, \ref{multideal}, and \ref{addition}, it now follows immediately that $\mathcal{WD}_{\ell_p}^{(\infty,\xi)}$ is an operator ideal.  In fact, the same combination of Propositions shows that $\mathcal{WD}_{\ell_p}^{(0,\xi)}$ is an operator ideal, but it turns out that, using Proposition \ref{VinWD} along with the fact that every family $\mathcal{S}_\xi$ contains all the singletons, regardless of our choice of $p$ or $\xi$ we always get $\mathcal{WD}_{\ell_p}^{(0,\xi)}=\mathcal{V}$, the completely continuous operators.

%\begin{proposition}\label{0givesV}Let $X$ and $Y$ be Banach spaces, and fix $1\leq p\leq\infty$ and $1\leq\xi<\omega_1$.  Then $\mathcal{WD}_{\ell_p}^{(0,\xi)}(X,Y)=\mathcal{V}(X,Y)$.\end{proposition}

%\begin{proof}From Proposition \ref{VinWD} we already know that $\mathcal{V}(X,Y)\subseteq\mathcal{WD}_{\ell_p}^{(0,\xi)}(X,Y)$, so we need only show that $\mathcal{WD}_{\ell_p}^{(0,\xi)}(X,Y)\subseteq\mathcal{V}(X,Y)$.    Indeed, let $T\in\mathcal{WD}_{\ell_p}^{(0,\xi)}(X,Y)$ and suppose towards a contradiction that $T\notin\mathcal{V}(X,Y)$.  Thus, there must exist a weakly null sequence $(x_n)$ such that $(Tx_n)$ is not norm-null.  By passing to a subsequence if necessary we can assume that $\norm{Tx_n}\geq\epsilon$ for all $n$ and some fixed $\epsilon>0$.  However by definition of $T\in\mathcal{WD}_{\ell_p}^{(0,\xi)}(X,Y)$ we can find a subsequence $(x_{n_k})$ such that $\norm{\sum\alpha_kTx_{n_k}}<\epsilon\norm{(\alpha_k)}_{\ell_1}$ for all $(\alpha_k)\in\mathcal{A}_\xi$.  So, due to $\mathcal{S}_\xi$ containing all the singletons, we get $\norm{Tx_{n_k}}<\epsilon$ for all $k$, in contradiction to the fact that $\norm{Tx_n}\geq\epsilon$ for all $n$.\end{proof}

%Of course, every normalized sequence has an upper $\ell_1$-estimate due to the triangle inequality for norms, and so when $p=1$ the ideal is trivially $\mathcal{WD}_{\ell_1}^{(\infty,\xi)}=\mathcal{L}$.  However, when $1<p\leq\infty$, let us observe that the ideal $\mathcal{WD}_{\ell_p}^{(\infty,\xi)}$ can get very small.

Let us now construct two important examples.

\begin{example}\label{compact}Let $X$ be a Banach space which fails to contain a copy of $\ell_1$.  (This is true in particular if $X$ is reflexive.)  Fix constants $1\leq q<p\leq\infty$ and $0\leq C\leq\infty$, and some countable ordinal $1\leq\xi<\omega_1$.  Then $\mathcal{WD}_{\ell_p}^{(C,\xi)}(X,\ell_q)=\mathcal{K}(X,\ell_q)$.\end{example}

\begin{proof}Assume $0\leq C<\infty$.  By Proposition \ref{VinWD} we already have $\mathcal{K}(X,\ell_q)\subseteq\mathcal{V}(X,\ell_q)\subseteq\mathcal{WD}_{\ell_p}^{(C,\xi)}(X,\ell_q)$, and so it suffices to prove $\mathcal{WD}_{\ell_p}^{(C,\xi)}(X,\ell_q)$ contains only compact operators.  For suppose towards a contradiction that there exists $T\in\mathcal{WD}_{\ell_p}^{(C,\xi)}(X,\ell_q)$ which is not compact.  Then we can find a seminormalized sequence $(x_n)\subset X$ for which $(Tx_n)$ fails to have a convergent subsequence.  Since $X$ fails to contain a copy of $\ell_1$, we can apply Rosenthal's $\ell_1$ Theorem to pass to a subsequence so that $(x_n)$ is weak Cauchy.  Hence we can pass to a further subsequence if necessary so that $(x_{2n}-x_{2n+1})$ and $(Tx_{2n}-Tx_{2n+1})$ are both weakly null and seminormalized.  This means the sequence $(x'_n)$ defined by $x'_n:=(x_{2n}-x_{2n+1})/\norm{x_{2n}-x_{2n+1}}$ is normalized and weakly null, whereas the sequence $(Tx'_n)$ is seminormalized and weakly null.  By passing to yet another subsequence if necessary, by Proposition 2.1.3 in \cite{AK06} we can assume $(Tx'_n)$ is $K$-equivalent, $K\geq 1$, to the unit vector basis of $\ell_q$.  Thus, by this equivalence together with the definition of $T\in\mathcal{WD}_{\ell_p}^{(C,\xi)}(X,Y)$, for any $\epsilon>0$ we can find a subsequence $(n_k)$ such that $\norm{(\alpha_k)}_{\ell_q}\leq K\norm{\sum\alpha_kTx'_{n_k}}\leq K(C+\epsilon)\norm{(\alpha_k)}_{\ell_p}$ for all $(\alpha_k)\in\mathcal{A}_\xi$.  Due to $\mathcal{S}_1\subseteq\mathcal{S}_\xi$, the above inequality holds also for all $(\alpha_k)\in\mathcal{A}_1$.  Notice that every $(\beta_k)\in c_{00}$ induces a corresponding ``spread out'' sequence $(\alpha_k)\in\mathcal{A}_1$ satisfying $\norm{(\alpha_k)}_{\ell_r}=\norm{(\beta_k)}_{\ell_r}$ for all $1\leq r\leq\infty$.  Thus we obtain the impossible estimate $\norm{(\beta_k)}_{\ell_q}=\norm{(\alpha_k)}_{\ell_q}\leq K(C+\epsilon)\norm{(\alpha_k)}_{\ell_p}=K(C+\epsilon)\norm{(\beta_k)}_{\ell_p}$ for all $(\beta_k)\in c_{00}$.\end{proof}

%Let us now construct an important example.

\begin{example}\label{whole-space}Let $X$ and $Y$ be Banach spaces, and fix numbers $1\leq p\leq q<\infty$ and an ordinal $1\leq\xi\leq\omega_1$.  Suppose $T\in\mathcal{L}(X,Y)$ is an operator such that $TX$ has a $K$-embedding, $K\geq 1$, into $\ell_q$.  (In other words, suppose there is an operator $Q\in\mathcal{L}(TX,\ell_q)$ which satisfies $K^{-1}\norm{y}\leq\norm{Qy}\leq K\norm{y}$ for all $y\in TX$.)  Then $T\in\mathcal{WD}_{\ell_p}^{(K^2\norm{T},\xi)}(X,Y)$, and the same result holds if $1\leq p\leq\infty$ and $TX$ has a $K$-embedding into $c_0$.  Thus, for $1\leq p\leq q<\infty$ we have $\mathcal{WD}_{\ell_p}^{(\infty,\xi)}(X,\ell_q)=\mathcal{L}(X,\ell_q)$, and for $1\leq p\leq\infty$ we have $\mathcal{WD}_{\ell_p}^{(\infty,\xi)}(X,c_0)=\mathcal{L}(X,c_0)$.\end{example}

\begin{proof}Fix a normalized weakly null sequence $(x_n)\subset X$, and denote by $Q\in\mathcal{L}(TX,\ell_q)$ (resp. $Q\in\mathcal{L}(TX,c_0)$) the $K$-embedding.  If $(Tx_n)$ contains a norm-null subsequence then we are done by Proposition \ref{norm-null}.  Otherwise let $\epsilon>0$, and find a subsequence so that $\norm{QTx_{n_k}}\to r$ with $0<r\leq K\norm{T}$, and quickly enough so that by the uniform version of Bessaga-Pe\l czy\'{n}ski combined with Lemma 2.1.1 in \cite{AK06}, we can pass to a further subsequence if necessary so that $(\frac{1}{r}QTx_{n_k})$ is $(1+\frac{\epsilon}{Kr})$-equivalent to the unit vector basis of $\ell_q$ (resp. $c_0$).  This gives us, in the $\ell_q$ case,

\begin{multline*}\norm{\sum\alpha_kTx_{n_k}}\leq K\norm{\sum\alpha_kQTx_{n_k}}=Kr\norm{\sum\alpha_k\frac{1}{r}QTx_{n_k}}\\\\\leq Kr\left(1+\frac{\epsilon}{Kr}\right)\norm{(\alpha_k)}_{\ell_q}\leq(K^2\norm{T}+\epsilon)\norm{(\alpha_k)}_{\ell_q}\leq(K^2\norm{T}+\epsilon)\norm{(\alpha_k)}_{\ell_p}\end{multline*}

\noindent for all $(\alpha_k)\in c_{00}$, and a similar inequality holds in the $c_0$ case.\end{proof}

We must lay some groundwork aimed at showing that, in case $\xi=1$, the class $\mathcal{WD}_{\ell_p}^{(\infty,1)}$ forms a Banach ideal.  We begin by defining a seminorm on the linear space $\mathcal{WD}_{\ell_p}^{(\infty,\xi)}(X,Y)$.

\begin{definition}Let $X$ and $Y$ be Banach spaces, and fix a constant $1\leq p\leq\infty$ and an ordinal $1\leq\xi\leq\omega_1$.  For each $T\in\mathcal{WD}_{\ell_p}^{(\infty,\xi)}(X,Y)$, we define $C_{(p,\xi)}(T):=\inf\{0\leq C<\infty:T\in\mathcal{WD}_{\ell_p}^{(C,\xi)}(X,Y)\}$.\end{definition}

%We must, of course, prove that $ C_{(p,\xi)}$ really is a seminorm on $\mathcal{WD}_{\ell_p}^{(\infty,\xi)}(X,Y)$.

\begin{proposition}\label{seminorm}Let $X$ and $Y$ be Banach spaces, and fix a constant $1\leq p\leq\infty$ and an ordinal $1\leq\xi\leq\omega_1$.  If $T\in\mathcal{WD}_{\ell_p}^{(\infty,\xi)}(X,Y)$ then $T\in\mathcal{WD}_{\ell_p}^{( C_{(p,\xi)}(T),\xi)}(X,Y)$.  Furthermore, $T\mapsto C_{(p,\xi)}(T)$ defines a seminorm on the linear space $\mathcal{WD}_{\ell_p}^{(\infty,\xi)}(X,Y)$.\end{proposition}

\begin{proof}The first part of the Proposition is clear from applying the definition of $T\in\mathcal{WD}_{\ell_p}^{(C,\xi)}(X,Y)$ for $ C_{(p,\xi)}(T)<C<\infty$, and absolute homogeneity is similarly obvious.  The only thing nontrivial to show is that the triangle inequality holds.  Indeed, let $S,T\in\mathcal{WD}_{\ell_p}^{(\infty,\xi)}(X,Y)$, and suppose $(x_n)$ is a normalized weakly null sequence.  Let $\epsilon>0$.  Then we can apply the definition of $S\in\mathcal{WD}_{\ell_p}^{( C_{(p,\xi)}(S),\xi)}(X,Y)$ to $(x_n)$ and $\frac{\epsilon}{2}>0$ to find a subsequence $(n_k)$ satisfying $\norm{\sum\alpha_kSx_{n_k}}\leq( C_{(p,\xi)}(S)+\frac{\epsilon}{2})\norm{(\alpha_k)}_{\ell_p}$  for all $(\alpha_k)\in\mathcal{A}_\xi$.  Then, we successively apply the definition of $T\in\mathcal{WD}_{\ell_p}^{( C_{(p,\xi)}(T),\xi)}(X,Y)$ to $(x_{n_k})$ and $\frac{\epsilon}{2}>0$ to find to a further subsequence $(n_{k_j})$ so that $\norm{\sum\alpha_jTx_{n_{k_j}}}\leq( C_{(p,\xi)}(T)+\frac{\epsilon}{2})\norm{(\alpha_j)}_{\ell_p}$ for all $(\alpha_j)\in\mathcal{A}_\xi$.  Thus, by these facts together with Lemma \ref{subsequence-hereditary}, %for each $(\alpha_j)\in\mathcal{A}_\xi$ we can find $(\beta_k)\in\mathcal{A}_\xi$ satisfying $\norm{(\beta_k)}_{\ell_p}=\norm{(\alpha_j)}_{\ell_p}$ and $\sum\alpha_jSx_{n_{k_j}}=\sum\beta_kSx_{n_k}$, so that

\begin{multline*}\norm{\sum\alpha_j(S+T)x_{n_{k_j}}}\leq\norm{\sum\alpha_jSx_{n_{k_j}}}+\norm{\sum\alpha_jTx_{n_{k_j}}}
% \\\\=\norm{\sum\beta_kSx_{n_k}}+\norm{\sum\alpha_jTx_{n_{k_j}}}\\\\\leq\left( C_{(p,\xi)}(S)+\frac{\epsilon}{2}\right)\norm{(\beta_k)}_{\ell_p}+\left( C_{(p,\xi)}(T)+\frac{\epsilon}{2}\right)\norm{(\alpha_j)}_{\ell_p}
\\\\=\left( C_{(p,\xi)}(S)+\frac{\epsilon}{2}\right)\norm{(\alpha_j)}_{\ell_p}+\left( C_{(p,\xi)}(T)+\frac{\epsilon}{2}\right)\norm{(\alpha_j)}_{\ell_p}\\\\=( C_{(p,\xi)}(S)+ C_{(p,\xi)}(T)+\epsilon)\norm{(\alpha_j)}_{\ell_p}.\end{multline*}

\noindent Thus, $C_{(p,\xi)}(S+T)\leq C_{(p,\xi)}(S)+ C_{(p,\xi)}(T)$, and we are done.\end{proof}

Next we show that $\mathcal{WD}_{\ell_p}^{(\infty,\xi)}$ fails to be norm-closed (as a class) whenever $p\neq 1$.  The main idea toward this end proceeds from the following Lemma.

\begin{lemma}\label{nonclosed-space}Fix constants $1<p\leq\infty$ and $1\leq q<\infty$, and an ordinal $1\leq\xi\leq\omega_1$.  Let $(X_m)$ and $(Y_m)$ be sequences of Banach spaces, and for each $m\in\mathbb{N}$, let $T_m\in\mathcal{WD}_{\ell_p}^{(\infty,\xi)}(X_m,Y_m)$ be an operator satisfying $\norm{T_m}=1$.  If $ C_{(p,\xi)}(T_m)\to\infty$ then there exists a subsequence $(m_j)$ and a sequence of operators $S_j\in\mathcal{WD}_{\ell_p}^{(\infty,\xi)}(X,Y)$ for which $S_j\to S\in\mathcal{L}(X,Y)$ but $S\notin\mathcal{WD}_{\ell_p}^{(\infty,\xi)}(X,Y)$, where we define $X:=(\bigoplus_{j=1}^\infty X_{m_j})_{\ell_q}$ and $Y:=(\bigoplus_{j=1}^\infty Y_{m_j})_{\ell_q}$.\end{lemma}

\begin{proof}Define the subsequence by letting $(m_j)$ be an increasing sequence satisfying $C_j:= C_{(p,\xi)}(T_{m_j})>j2^j$ for every $j$.  Next, set $S:=\bigoplus_{j=1}^\infty 2^{-j}T_{m_j}\in\mathcal{L}(X,Y)$.  For each $i$, let $P_i\in\mathcal{L}(Y)$ denote the continuous linear projection onto the first $i$ coordinates of $Y$, and set $S_i:=P_iS$.  It's easy to see that $S_i\to S$.  %Indeed let $(x_j)\in X$.  Then

%\begin{multline*}\norm{(S-S_i)(x_j)}_Y=\left(\sum_{j=i+1}^\infty 2^{-jq}\norm{T_{m_j}x_j}_{Y_j}^q\right)^{1/q}\\\\\leq\left(\sum_{j=i+1}^\infty 2^{-jq}\norm{x_j}_{X_j}^q\right)^{1/q}\leq 2^{-i-1}\norm{(x_j)}_X\end{multline*}

%\noindent so that $\norm{S-S_i}\leq 2^{-i-1}\to 0$ as desired.

Next, we claim that each $S_i\in\mathcal{WD}_{\ell_p}^{(M_i,\xi)}(X,Y)\subseteq\mathcal{WD}_{\ell_p}^{(\infty,\xi)}(X,Y)$, where we set $M_i:=\norm{(2^{-j}C_j)_{j=1}^i}_{\ell_q}$.  Indeed fix any $i\in\mathbb{N}$, and let $(x_n)$ be a normalized weakly null sequence in $X$.  Pick any $\epsilon>0$.  For each $j$, let $\widetilde{X}_j$ be the obvious isometrically isomorphic copy of $X_j$ contained in $X$, and let $U_j:\widetilde{X}_j\to X_j$ be the corresponding isometric isomorphism.  For each $n$, write $x_n=(x_{n,j})_j\in X$.  Then $(x_{n,j})_n$ is a sequence in $X_j$ which is bounded by 1.  If $(x_{n,j})_n$ has a norm-null subsequence, then by Proposition \ref{norm-null} we can find a subsequence $(n_k)$ such that, for all $(\alpha_k)\in c_{00}$,

\begin{equation}\label{inequality}\norm{\sum_k\alpha_kT_{m_j}x_{n_k,j}}\leq\left(C_j+\frac{\epsilon 2^j}{i^{1/q}}\right)\norm{(x_{n_k,j})_k}_{\ell_p}\end{equation}

\noindent Otherwise we can find a subsequence $(n_k)$ so that $\norm{x_{n_k,j}}_{X_j}\to r$ as $k\to\infty$ for some $0<r\leq 1$.  Clearly, $(x_{n,j})_n$ is weakly null in $X_j$, and so by Propositions \ref{seminormalized-tends-to-1} and \ref{seminorm}, we can pass to a further subsequence if necessary so that, for all $(\alpha_k)\in\mathcal{A}_\xi$,

%We claim that $(x_{n,j})_n$ is weakly null.  Indeed let $f\in X_j^*$, and define the norm-1 linear operator $\pi_j\in\mathcal{L}(X,X_j)$ by $\pi_j:(x_i)\mapsto x_j$.  Then $\pi_j^*f\in X^*$ so that $f(x_{n,j})=(\pi_j^*f)(x_n)\to 0$ as $n\to\infty$.  Thus $(x_{n,j})_n$ is weakly null as claimed, and hence so too is $(\frac{1}{r}x_{n_k,j})_k$.

\begin{multline*}\norm{\sum_k\alpha_kT_{m_j}x_{n_k,j}}=r\norm{\sum_k\alpha_kT_{m_j}\frac{x_{n_k,j}}{r}}\\\\\leq r\left(C_j+\frac{\epsilon 2^j}{i^{1/q}r}\right)\norm{(x_{n_k,j})_k}_{\ell_p}\leq\left(C_j+\frac{\epsilon 2^j}{i^{1/q}}\right)\norm{(x_{n_k,j})_k}_{\ell_p}.\end{multline*}

\noindent In either case, for each $j$ and any subsequence of $(x_{n,j})_n$, we can pass to a further subsequence so that the inequality \eqref{inequality} holds for all $(\alpha_k)\in\mathcal{A}_\xi$.

Thus, by successively passing to further subsequences for $j=1,\cdots,i$, due to Lemma \ref{subsequence-hereditary}, we get a subsequence $(n_k)$ such that \eqref{inequality} holds for all $j=1,\cdots,i$ and all $(\alpha_k)\in\mathcal{A}_\xi$.  In particular, this means

\begin{multline*}\norm{\sum_k\alpha_kS_ix_{n_k}}=\left(\sum_{j=1}^i2^{-jq}\norm{\sum_k\alpha_kT_{m_j}x_{n_k,j}}^q\right)^{1/q}\\\\\leq\left(\sum_{j=1}^i2^{-jq}\left(C_j+\frac{\epsilon 2^j}{i^{1/q}}\right)^q\norm{(\alpha_k)}_{\ell_p}^q\right)^{1/q}=\norm{\left(2^{-j}C_j+\frac{\epsilon}{i^{1/q}}\right)_{j=1}^i}_{\ell_q}\norm{(\alpha_k)}_{\ell_p}\\\\\leq\left(\norm{\left(2^{-j}C_j\right)_{j=1}^i}_{\ell_q}+\norm{\left(\frac{\epsilon}{i^{1/q}}\right)_{j=1}^i}_{\ell_q}\right)\norm{(\alpha_k)}_{\ell_p}=(M_i+\epsilon)\norm{(\alpha_k)}_{\ell_p},\end{multline*}

\noindent which proves the claim that $S_i\in\mathcal{WD}_{\ell_p}^{(M_i,\xi)}(X,Y)\subseteq\mathcal{WD}_{\ell_p}^{(\infty,\xi)}(X,Y)$.

However, it cannot be that $S\in\mathcal{WD}_{\ell_p}^{(C,\xi)}(X,Y)$ for any $0\leq C<\infty$.  To show why not, fix $i\in\mathbb{N}$, and let $(x_n)$ be a normalized weakly null sequence in $X_i$.  Then let $\epsilon>0$ be such that, for any subsequence $(n_k)$, there exists $(\alpha_k)\in\mathcal{A}_\xi$ with $\norm{\sum_k\alpha_kT_{m_i}x_{n_k}}>(i2^i+\epsilon)\norm{(\alpha_k)}_{\ell_p}$.  Let $Q_i:X_i\to X$ be the canonical embedding of $X_i$ into $X$, and observe that $(Q_ix_n)_n$ is a normalized weakly null sequence in $X$.  However, for every subsequence $(n_k)$ there exists $(\alpha_k)\in\mathcal{A}_\xi$ with

%Since $Q_i^*f\in X_i^*$ with $(x_n)$ weakly null, we get $f(Q_ix_n)=(Q_i^*f)(x_n)\to 0$ as $n\to\infty$, for every $f\in X^*$.  Hence,

\begin{multline*}\norm{\sum_k\alpha_kSQ_ix_{n_k}}=2^{-i}\norm{\sum_k\alpha_kT_{m_i}x_{n_k}}\\>2^{-i}(i2^i+\epsilon)\norm{(\alpha_k)}_{\ell_p}\geq(i+2^{-i}\epsilon)\norm{(\alpha_k)}_{\ell_p}.\end{multline*}

\noindent It follows that $S\notin\mathcal{WD}_{\ell_p}^{(i,\xi)}(X,Y)$ for any $i$, and hence $S\notin\mathcal{WD}_{\ell_p}^{(\infty,\xi)}(X,Y)$.\end{proof}

%Now, we are ready to show that $\mathcal{WD}_{\ell_p}^{(\infty,\xi)}(X,Y)$ fails to always be norm-closed in $\mathcal{L}(X,Y)$.

\begin{example}\label{not-norm-closed-example}Fix a constant $1<p\leq\infty$ and an ordinal $1\leq\xi\leq\omega_1$.  There exists a Banach space $X$ for which $\mathcal{WD}_{\ell_p}^{(\infty,\xi)}(X)$ is not norm-closed.  If $p\neq\infty$, then we can choose $X$ to be reflexive.\end{example}

\begin{proof}For convenience in writing, let us consider the case where $p\neq\infty$.  The case where $p=\infty$ uses $c_0$ in place of $\ell_p$, and the resulting proof is nearly identical, except that the resulting space $X$ is not reflexive.

Let $(e_n)$ denote the unit vector basis of $\ell_p$.  For each finite $E\subset\mathbb{N}$, define the functional $f_E\in\ell_p^*$ by the rule $f_E(e_n)=1$ if $n\in E$ and $f_E(e_n)=0$ otherwise.  Now, fix $m\in\mathbb{N}$, and define the norming set $\mathcal{B}_m:=B_{\ell_p^*}\cup\{f_E:E\subset\mathbb{N},\#E=m\}$, where $B_{\ell_p^*}$ denotes the closed unit ball of $\ell_p^*=\ell_q$.  Notice that for every $E\subset\mathbb{N}$ of size $m$, we have $|f_E(\sum\alpha_ke_k)|\leq\norm{(\alpha_k)_{k\in E}}_{\ell_1}\leq m^{1-\frac{1}{p}}\norm{(\alpha_k)}_{\ell_p}$ so that $\norm{f_E}_{\ell_p^*}\leq m^{1-\frac{1}{p}}$.  So $\mathcal{B}_m$ is a bounded subset of $\ell_p^*$ containing $B_{\ell_p^*}$.  Due to the identity $\norm{x}_{\ell_p}=\sup_{f\in B_{\ell_p^*}}|f(x)|$, we can now define an equivalent norm $\norm{\cdot}_m$ on $\ell_p$ by the rule $\norm{x}_m:=\sup_{f\in\mathcal{B}_m}|f(x)|$.  Put $X_m:=(\ell_p,\norm{\cdot}_m)$, and notice that for all $n$ and $E$, we have $|f_E(e_n)|\leq 1$.  Hence $(e_n)$ is still normalized in $X_m$, and weakly null since $X_m$ is isomorphic to $\ell_p$.  Furthermore, this isomorphism also means the identity map $I_m\in\mathcal{L}(X_m)$ is a norm-1 operator which lies in $\mathcal{WD}_{\ell_p}^{(\infty,\xi)}(X_m)$ by Example \ref{whole-space}.  However, we will show that $ C_{(p,\xi)}(I_m)\geq m^{1-\frac{1}{p}}$.

Suppose $C<m^{1-\frac{1}{p}}$, and let $0<\epsilon<m^{1-\frac{1}{p}}-C$.  Then, let $(e_{n_k})$ be any subsequence of $(e_n)$, which we have previously observed is normalized and weakly null in $X_m$.  Pick $E=(m+1<m+2<\cdots<2m)\in\mathcal{S}_1\subseteq\mathcal{S}_\xi$ of size $m$, and define $F:=(n_{m+1}<n_{m+2}<\cdots<n_{2m})$.  Since $\mathcal{S}_\xi$ is spreading, we have $F\in\mathcal{S}_\xi$, and also of size $m$.  Next, define $(\alpha_k)\in\mathcal{A}_\xi$ by letting $\alpha_k=1$ for all $k\in E$ and $\alpha_k=0$ otherwise.  Then

\begin{multline*}\norm{\sum\alpha_ke_{n_k}}_m\geq\left|f_F\left(\sum\alpha_ke_{n_k}\right)\right|=\left|f_F\left(\sum_{n\in F}e_n\right)\right|=m\\\\=m^{1-\frac{1}{p}}\norm{(\alpha_k)}_{\ell_p}>(C+\epsilon)\norm{(\alpha_k)}_{\ell_p}.\end{multline*}

\noindent Thus, the identity map $I_m$ does not lie in $\mathcal{WD}_{\ell_p}^{(C,\xi)}(X_m)$, and $ C_{(p,\xi)}(I_m)\geq m^{1-\frac{1}{p}}$ as claimed.

We have therefore constructed a sequence $(X_m)$ of Banach spaces and a corresponding sequence $I_m\in\mathcal{WD}_{\ell_p}^{(\infty,\xi)}(X_m)$ of norm-1 operators with $C_{(p,\infty)}(I_m)\to\infty$.  By Lemma \ref{nonclosed-space}, there exists a space $X$ for which $\mathcal{WD}_{\ell_p}^{(\infty,\xi)}(X)$ fails to be norm-closed, and in case $p\neq\infty$, we can choose it to be reflexive.\end{proof}

Even though $\mathcal{WD}_{\ell_p}^{(\infty,\xi)}$ is not a norm-closed operator ideal, when $\xi=1$ its components are $F_\sigma$-subsets of $\mathcal{L}(X,Y)$, as the following Proposition shows.

\begin{proposition}\label{norm-closed}Let $X$ and $Y$ be Banach spaces, and fix constants $0\leq C<\infty$ and $1\leq p\leq\infty$.  We consider the case $\xi=1$.  Then $\mathcal{WD}_{\ell_p}^{(C,1)}(X,Y)$ is a norm-closed subset of $\mathcal{L}(X,Y)$.\end{proposition}

\begin{proof}Let $(T_j)$ be a sequence in $\mathcal{WD}_{\ell_p}^{(C,1)}(X,Y)$ converging to some $T\in\mathcal{L}(X,Y)$.  Fix any $\epsilon>0$, and let $(x_n)\subset X$ be a normalized weakly null sequence in $X$.  Without loss of generality we may assume $\norm{T-T_j}<\epsilon/(2j^{1-\frac{1}{p}})$ for all $j$.  Let $(x_{n_k})$ be a subsequence formed by a diagonal argument using the $T_j$'s with $\frac{\epsilon}{2}>0$.  In other words, begin with a subsequence $(x_{n_{1,k}})$ given by the definition of $T_1\in\mathcal{WD}_{\ell_p}^{(C,1)}(X,Y)$, corresponding to $\frac{\epsilon}{2}>0$.  Then find a \emph{further} subsequence $(x_{n_{2,k}})\subset(x_{n_{1,k}})$ given by the definition of $T_2\in\mathcal{WD}_{\ell_p}^{(C,\xi)}(X,Y)$, also corresponding to $\frac{\epsilon}{2}>0$, and so on.  Finally, for each $k$, put $n_k:=n_{k,k}$.

Let $(\alpha_k)\in\mathcal{A}_1$, and set $m:=\min\text{supp}(\alpha_k)\leq\#\text{supp}(\alpha_k)$.  Notice that $(x_{n_k})_{k\geq m}$ is a subsequence of $(x_{n_{m,i}})_{i\geq m}$ so that by Lemma \ref{subsequence-hereditary}, % we can write $\sum_k\alpha_kT_mx_{n_k}=\sum_i\beta_iT_mx_{n_{m,i}}$ for some $(\beta_i)\in\mathcal{A}_1$ satisfying $\norm{(\beta_i)}_{\ell_p}=\norm{(\alpha_k)}_{\ell_p}$.  Thus,

\begin{multline*}\norm{\sum\alpha_kTx_{n_k}}\leq\norm{\sum_k\alpha_kT_mx_{n_k}}+\norm{T-T_m}\norm{\sum\alpha_kx_{n_k}}
% \\\\=\norm{\sum_i\beta_i T_mx_{n_{m,i}}}+\norm{T-T_m}\norm{(\alpha_k)}_{\ell_1}\\\\\leq\left(C+\frac{\epsilon}{2}\right)\norm{(\beta_i)}_{\ell_p}+m^{1-\frac{1}{p}}\norm{T-T_m}\norm{(\alpha_k)}_{\ell_p}
\\\\<\left(C+\frac{\epsilon}{2}\right)\norm{(\alpha_k)}_{\ell_p}+m^{1-\frac{1}{p}}\left(\frac{\epsilon}{2m^{1-\frac{1}{p}}}\right)\norm{(\alpha_k)}_{\ell_p}=(C+\epsilon)\norm{(\alpha_k)}_{\ell_p}\end{multline*}\end{proof}

\begin{definition}Fix Banach spaces $X$ and $Y$, along with a constant $1\leq p\leq\infty$ and an ordinal $1\leq\xi\leq\omega_1$.  We define the norm $\norm{\cdot}_{(p,\xi)}$ on the space $\mathcal{WD}_{\ell_p}^{(\infty,\xi)}(X,Y)$ by the rule $\norm{T}_{(p,\xi)}:=\norm{T}_{\mathcal{L}(X,Y)}+ C_{(p,\xi)}(T)$.\end{definition}

Notice that $\norm{\cdot}_{(p,\xi)}$ is indeed a norm on $\mathcal{WD}_{\ell_p}^{(\infty,\xi)}(X,Y)$, as it is the sum of a norm and a seminorm.  %Propositions \ref{seminorm} and \ref{norm-closed} also give us the following.

%\begin{proposition}\label{complete-norm}Fix Banach spaces $X$ and $Y$, along with a constant $1\leq p\leq\infty$ and an ordinal $1\leq\xi\leq\omega_1$.  Then $\mathcal{WD}_{\ell_p}^{(\infty,\xi)}(X,Y)$ is a Banach space under the norm $\norm{\cdot}_{(p,\xi)}$.\end{proposition}

%\begin{proof}Suppose $(T_n)$ is a $\norm{\cdot}_{(p,\xi)}$-Cauchy sequence.  Then it is $\norm{\cdot}_{(p,\xi)}$-bounded and hence $C_{(p,\xi)}$-bounded, say by $M>0$.  It is also Cauchy in the operator norm so that $T_n\to T$ for some $T\in\mathcal{L}(X,Y)$.  By Proposition \ref{seminorm}, every $T_n$ lies in the set $\mathcal{WD}_{\ell_p}^{(M,\xi)}(X,Y)$, which is closed under the operator norm by Proposition \ref{norm-closed}.  Hence, $T\in\mathcal{WD}_{\ell_p}^{(\infty,\xi)}(X,Y)$ as well.\end{proof}

\begin{proposition}\label{banach-ideal}Fix $1\leq p\leq\infty$.  In case $\xi=1$, the rule $\norm{\cdot}_{(p,1)}$ is an ideal norm which makes $\mathcal{WD}_{\ell_p}^{(\infty,1)}$ into a Banach ideal.\end{proposition}

\begin{proof}  As was observed earlier, that $\mathcal{WD}_{\ell_p}^{(\infty,1)}$ is an operator ideal follows from Propositions \ref{VinWD}, \ref{multideal}, and \ref{addition}.  To show that $\norm{\cdot}_{(p,1)}$ induces a complete norm on each component space $\mathcal{WD}_{\ell_p}^{(\infty,1)}(X,Y)$, suppose $(T_n)$ is a $\norm{\cdot}_{(p,1)}$-Cauchy sequence.  Then it is $\norm{\cdot}_{(p,1)}$-bounded and hence $C_{(p,1)}$-bounded, say by $M>0$.  It is also Cauchy in the operator norm so that $T_n\to T$ for some $T\in\mathcal{L}(X,Y)$.  By Proposition \ref{seminorm}, every $T_n$ lies in the set $\mathcal{WD}_{\ell_p}^{(M,1)}(X,Y)$, which is closed under the operator norm by Proposition \ref{norm-closed}.  Hence, $T\in\mathcal{WD}_{\ell_p}^{(\infty,1)}(X,Y)$ as well, and it remains only to show that $\norm{\cdot}_{(p,1)}$ is indeed an ideal norm.

Since any element of the form $x^*\otimes y$ is rank-1, it is completely continuous.  By Proposition \ref{VinWD}, this means $C_{(p,1)}(x^*\otimes y)=0$ and hence $\norm{x^*\otimes y}_{(p,1)}=\norm{x^*\otimes y}_{\mathcal{L}(X,Y)}=\sup_{x\in S_X}\norm{(x^*\otimes y)(x)}_Y=\sup_{x\in S_X}|x^*(x)|\norm{y}_Y=\norm{x^*}_{X^*}\norm{y}_Y$.  The triangle inequality follows from the fact that $\norm{\cdot}_{(p,1)}$ is a norm on each component space $\mathcal{WD}_{\ell_p}^{(\infty,1)}(X,Y)$.  That $\norm{BTA}_{(p,1)}\leq\norm{B}_{\mathcal{L}(Y,Z)}\norm{T}_{(p,1)}\norm{A}_{\mathcal{L}(W,X)}$ for all $T\in\mathcal{J}(X,Y)$, $A\in\mathcal{L}(W,X)$, and $B\in\mathcal{L}(Y,Z)$, follows naturally from Propositions \ref{multideal} and \ref{seminorm}.\end{proof}

\section{Significance of parameters}

%In this section, we discuss the significance of the parameters $p$ and $\xi$ in the construction of the ideals $\mathcal{WD}_{\ell_p}^{(\infty,\xi)}$.  Let's begin by observing that for any fixed $1\leq\xi\leq\omega_1$, the ideals $\mathcal{WD}_{\ell_p}^{(\infty,\xi)}$ are all distinct as $p$ ranges over $1\leq p\leq\infty$.

Let $1<q<p<\infty$.  By Example \ref{compact} we get $\mathcal{WD}_{\ell_p}^{(\infty,\xi)}(\ell_q)=\mathcal{K}(\ell_q)$, whereas by Example \ref{whole-space} we get $\mathcal{WD}_{\ell_q}^{(\infty,\xi)}(\ell_q)=\mathcal{L}(\ell_q)$.  Applying Proposition \ref{inclusion} therefore gives us the following.

\begin{proposition}\label{p-distinct}Fix any ordinal $1\leq\xi\leq\omega_1$.  For $1\leq q<p\leq\infty$, the norm-closed operator ideals $\overline{\mathcal{WD}}_{\ell_p}^{(\infty,\xi)}$ and $\overline{\mathcal{WD}}_{\ell_q}^{(\infty,\xi)}$ are distinct (as classes).\end{proposition}

\noindent However, it is natural to also ask whether the classes $\mathcal{WD}_{\ell_p}^{(\infty,\xi)}$ are distinct as $\xi$ ranges over $1\leq\xi\leq\omega_1$.  Obviously, this is not the case for the trivial ideal $\mathcal{WD}_{\ell_1}^{(\infty,\xi)}=\mathcal{L}$.  The question remains open in general for $1<p\leq\infty$, but in this section we do give a {\it partial} answer by showing that $\mathcal{WD}_{\ell_p}^{(\infty,\xi)}\neq\mathcal{WD}_{\ell_p}^{(\infty,\omega_1)}$ for all countable ordinals $1\leq\xi<\omega_1$.  Our task requires a few preliminaries, given below.

\begin{proposition}\label{ordinal-inclusion}Let $X$ and $Y$ be Banach spaces, and fix a constant $1\leq p\leq\infty$.  Let $1\leq\xi<\zeta\leq\omega_1$ be ordinals, and $0\leq C\leq\infty$.  Then $\mathcal{WD}_{\ell_p}^{(C,\zeta)}(X,Y)\subseteq\mathcal{WD}_{\ell_p}^{(C,\xi)}(X,Y)$.\end{proposition}

\begin{proof}We may assume $C\neq\infty$.  Suppose $T\in\mathcal{WD}_{\ell_p}^{(C,\zeta)}(X,Y)$, and let $(x_n)$ be a normalized weakly null sequence in $X$, and $\epsilon>0$.  Then there exists a subsquence $(n_k)$ such that $\norm{\sum\alpha_kTx_{n_k}}\leq(C+\epsilon)\norm{(\alpha_k)}_{\ell_p}$ for all $(\alpha_k)\in\mathcal{A}_\zeta$.  Let $d=d(\xi,\zeta)$ be such that if $E\in\mathcal{S}_\xi$ with $\min E\geq d$ then $E\in\mathcal{S}_\zeta$.  Now, let $(\alpha_k)\in\mathcal{S}_\xi$, and define the scalar sequence $(\beta_k)$ as $\beta_k:=\alpha_{k-d}$ for $k\geq d$ and $\beta_k:=0$ for $k<d$.  By the spreading property of $\mathcal{S}_\xi$ we have $(\beta_k)\in\mathcal{A}_\xi$, and since also $\min\text{supp}(\beta_k)\geq d$ we have $(\beta_k)\in\mathcal{A}_\zeta$.  Thus, $\norm{\sum\alpha_kTx_{n_{k+d}}}=\norm{\sum\beta_kTx_{n_k}}\leq(C+\epsilon)\norm{(\beta_k)}_{\ell_p}=(C+\epsilon)\norm{(\alpha_k)}_{\ell_p}$.\end{proof}

%Now, let's lay some groundwork for showing that $\mathcal{WD}_{\ell_p}^{(\infty,\xi)}\neq\mathcal{WD}_{\ell_p}^{(\infty,\omega_1)}$ whenever $1<p\leq\infty$ and $1\leq\xi<\omega_1$.  We will do it by looking at the duals of the Tsirelson-type spaces and their $q$-convexifications.  So, first we will need to recall some basic facts about those spaces.

Let $1\leq\xi<\omega_1$ be a countable ordinal.  A finite sequence $(E_i)_{i=1}^j$ of finite subsets of $\mathbb{N}$ is called {\bf $\boldsymbol{\mathcal{S}_\xi}$-admissible} whenever $E_1<\cdots<E_j$ and $\{\min E_i\}_{i=1}^j\in\mathcal{S}_\xi$.  Then the {\bf Tsirelson-type space} $T[\frac{1}{2},\mathcal{S}_\xi]$ is the completion of $c_{00}$ under the norm $\norm{\cdot}_T$ uniquely defined by the implicit equation $\norm{x}_T=\max\{\norm{x}_{\ell_\infty},\frac{1}{2}\sup\sum_i\norm{E_ix}_T\}$, where the ``sup'' is taken over all $j\in\mathbb{N}$ and all $\mathcal{S}_\xi$-admissible families $(E_i)_{i=1}^j$.  Here we use the notation $E_ix:=\sum_{n\in E_i}\alpha_ne_n$ for $x:=\sum\alpha_ne_n\in c_{00}$, where $(e_n)$ are the canonical basis vectors in $c_{00}$.  We also use the abbreviation $T=T[\frac{1}{2},\mathcal{S}_\xi]$ when the ordinal $\xi$ is understood from context.

It is easily seen that the canonical unit vectors in $c_{00}$ form a normalized 1-unconditional basis for $T$.  For $1\leq q<\infty$, its {\bf $\boldsymbol{q}$-convexification} $T_q[\frac{1}{2},\mathcal{S}_\xi]$, where we again use the abbreviation $T_q=T_q[\frac{1}{2},\mathcal{S}_\xi]$, is then usually defined in the literature by setting $T_q=T_q[\textstyle\frac{1}{2},\mathcal{S}_\xi]=\{(\alpha_n):(|\alpha_n|^q)\in T=T[\frac{1}{2},\mathcal{S}_\xi]\}$, which is a Banach space under the norm $\norm{(\alpha_n)}_{T_q}:=\norm{(|\alpha_n|^q)}_T^{1/q}$, and with the canonical unit vectors in $c_{00}$ again forming a normalized 1-unconditional basis.  (Notice also that if $q=1$ then we get $T_1=T$.)  However, it will serve our purposes much better to use instead the following equivalent construction (cf. \cite{JL03}, p1062).  We inductively define a sequence $(\norm{\cdot}_n)$ of norms on $c_{00}$.  Set $\norm{\cdot}_0:=\norm{\cdot}_{\ell_\infty}$ and define each successive $\norm{\cdot}_{n+1}$ by the rule $\norm{x}_{n+1}=\max\{\norm{x}_{\ell_\infty},2^{-1/q}\sup(\sum_i\norm{E_ix}_n^q)^{1/q}\}$, where the ``sup'' is taken over all $j\in\mathbb{N}$ and all $\mathcal{S}_\xi$-admissible families $(E_i)_{i=1}^j$.  Then $\norm{x}_{T_q}:=\lim_{n\to\infty}\norm{x}_n$ defines a norm on $c_{00}$.  In fact, it is easily seen (cf., e.g., the kind of argument used in the proof to Theorem 10.3.2 in \cite{AK06}) that $\norm{\cdot}_{T_q}$ is the unique norm on $c_{00}$ satisfying the implicit equation $\norm{x}_{T_q}=\max\{\norm{x}_{\ell_\infty},2^{-1/q}\sup(\sum_i\norm{E_ix}_{T_q}^q)^{1/q}\}$.  The space $T_q$ is just the completion of $c_{00}$ under this norm.

Due to this construction, $\norm{x}_{T_q}\leq\norm{x}_{\ell_q}$ for each $x\in c_{00}$.  Furthermore, $T_q$ is known to be a reflexive Banach space which contains no copy of $\ell_q$.  When $q=1$ this follows from Proposition 5.1 in \cite{AA92}.  In case $1<q<\infty$, Remark T.1 on p1062 of \cite{JL03} tells us that $T_q$ is an asymptotic $\ell_q$ space which contains no copy of $\ell_q$, and thus by Remark 6.3 in \cite{MT93} it is also reflexive.  Therefore each dual space $T_q^*$ is a reflexive space which fails to contain any copy of $\ell_p$, $\frac{1}{p}+\frac{1}{q}=1$.  Notice that $T_q^*$ can also be viewed as a completion of $c_{00}$ under some norm $\norm{\cdot}_{T_q^*}$, with the usual action $f(x)=\sum\alpha_n\beta_n$ for $f=(\alpha_n)\in T_q^*$ and $x=(\beta_n)\in T_q$.

%\begin{remark}Notice that $T_q^*$ can also be viewed as a sequence space with the usual action $f(x)=\sum_{n=1}^\infty\alpha_n\beta_n$ for $f=(\alpha_n)\in T_q^*$ and $x=(\beta_n)\in T_q$.  Indeed, the map $f\in T_q^*\mapsto(f(e_n))\in\ell_\infty$ is clearly injective so that each $f\in T_q^*$ has a unique representation as an element of $\ell_\infty$.  In particular, under this identification, the $(e_n)\in T_q^*$ are the biorthogonal functionals for $(e_n)\in T_q$, where $(e_n)$ denotes the standard unit vectors in $c_{00}$.  Recall that the biorthogonal functionals to a basis of any reflexive Banach space form a basis for the dual space (cf., e.g., Theorems 3.2.12 and 3.2.13 in \cite{AK06}).  Notice also that for any $e_n\in T_q^*$ and $x\in T_q$ we have $|e_n(x)|\leq\norm{x}_{\ell_\infty}\leq\norm{x}_{T_q}$.  It follows that $(e_n)$ is a normalized basis for $T_q^*$.\end{remark}

In \cite{OA13} was given an implicit formula for the norm of $T_1[\frac{1}{2},\mathcal{S}_1]^*$.  It is natural to conjecture that a similar formula always holds for the norm of $T_q[\frac{1}{2},\mathcal{S}_\xi]^*$, but for our purposes we only need a crude estimate.

\begin{lemma}\label{convexified-estimate}Let $1<p\leq\infty$ and $1\leq q<\infty$ be conjugate, i.e. $\frac{1}{p}+\frac{1}{q}=1$.  Set $1\leq\xi<\omega_1$ and $T_q=T_q[\frac{1}{2},\mathcal{S}_\xi]$.  Then $\norm{x}_{T_q^*}\leq 2^{1/q}\norm{(\norm{E_ix}_{T_q^*})}_{\ell_p}$ for all $x\in c_{00}$ and $\mathcal{S}_\xi$-admissible families $(E_i)_{i=1}^j$ satisfying $x=\sum_{i=1}^jE_ix$.\end{lemma}

\begin{proof}Let $y\in T_q$.  Since $x=\sum_{i=1}^jE_ix$ we have $x(y)=\sum_{i=1}^j(E_ix)(E_iy)$.  Then by this fact together with H\"older and the relation $2^{-1/q}(\sum_{i=1}^j\norm{E_iy}_{T_q}^q)^{1/q}\leq\norm{y}_{T_q}$ (which follows from the construction of $T_q$), we have

\begin{multline*}|x(y)|=|\sum_{i=1}^j(E_ix)(E_iy)|\leq\sum_{i=1}^j|(E_ix)(E_iy)|\leq\sum_{i=1}^j\norm{E_ix}_{T_q^*}\norm{E_iy}_{T_q}\\\\\leq\norm{(\norm{E_ix}_{T_q^*})}_{\ell_p}\left(\sum_{i=1}^j\norm{E_iy}_{T_q}^q\right)^{1/q}\leq 2^{1/q}\norm{(\norm{E_ix}_{T_q^*})}_{\ell_p}\norm{y}_{T_q}.\end{multline*}\end{proof}

%We need two more estimates before we can proceed to the main example.

\begin{lemma}\label{Tsirelson-block-estimate}Let $T_q=T_q[\frac{1}{2},\mathcal{S}_\xi]$, $1\leq q<\infty$ and $1\leq\xi<\omega_1$, and let $(u_k)$ be any normalized block basic sequence in the dual space $T_q^*$ (with respect to the canonical unit vectors in $c_{00}$).  Then for every $(\alpha_k)\in\mathcal{A}_\xi$ we have $\norm{\sum\alpha_ku_k}_{T_q^*}\leq 2^{1/q}\norm{(\alpha_k)}_{\ell_p}$, where $1<p\leq\infty$ is conjugate to $q$, that is, $\frac{1}{p}+\frac{1}{q}=1$.\end{lemma}

\begin{proof}Write $\text{supp}(\alpha_k)=:\{k_1,\cdots,k_j\}\in\mathcal{S}_\xi$, and set $E_i:=\text{supp}\{u_{k_i}\}$ for each $1\leq i\leq j$.  Then $x:=\sum\alpha_ku_k=\sum_{i=1}^jE_ix$, where $\norm{E_ix}_{T_q^*}=\norm{\alpha_{k_i}u_{k_i}}_{T_q^*}=|\alpha_{k_i}|$ for each $1\leq i\leq j$.  Furthermore, due to $k_i\leq\min\text{supp}\{u_{k_i}\}=\min E_i$ together with $\{k_1,\cdots,k_j\}\in\mathcal{S}_\xi$ and the spreading property of Schreier families, we see that $(E_i)_{i=1}^j$ is $\mathcal{S}_\xi$-admissible.  All of this together with Lemma \ref{convexified-estimate} means $\norm{\sum\alpha_ku_k}_{T_q^*}\leq 2^{1/q}\norm{(\norm{E_ix}_{T_q^*})_{i=1}^j}_{\ell_p}=2^{1/q}\norm{(\alpha_{k_i})_{i=1}^j}_{\ell_p}=2^{1/q}\norm{(\alpha_k)}_{\ell_p}$.\end{proof}

\begin{lemma}\label{convexified-lower-estimate}Set $T_q=T_q[\frac{1}{2},\mathcal{S}_\xi]$, $1\leq q<\infty$ and $1\leq\xi<\omega_1$.  Let $1<p\leq\infty$ denote the conjugate of $q$, that is, $\frac{1}{p}+\frac{1}{q}=1$.  Then $\norm{x^*}_{\ell_p}\leq\norm{x^*}_{T_q^*}$ for all $x^*\in c_{00}$.\end{lemma}

\begin{proof}Since $c_{00}\subseteq\ell_p=\ell_q^*$ with $c_{00}$ dense in $\ell_q$, for each $\epsilon>0$ we can find $x\in c_{00}$ such that $|x^*(x)|\geq(\norm{x^*}_{\ell_p}-\epsilon)\norm{x}_{\ell_q}$.  Combining this with the relation $\norm{x}_{T_q}\leq\norm{x}_{\ell_q}$ (which follows from the construction of $T_q$) we get $|x^*(x)|\geq(\norm{x^*}_{\ell_p}-\epsilon)\norm{x}_{\ell_q}\geq(\norm{x^*}_{\ell_p}-\epsilon)\norm{x}_{T_q}$ and hence $\norm{x^*}_{T_q^*}\geq\norm{x^*}_{\ell_p}-\epsilon$.  Letting $\epsilon\to 0$ completes the proof.\end{proof}

%Now we are ready to prove the main result of this section.

\begin{example}\label{ordinal-distinct}Set $T_q=T_q[\frac{1}{2},\mathcal{S}_\xi]$, $1\leq q<\infty$ and $1\leq\xi<\omega_1$, and let $T_q^*$ denote its dual.  Let $1<p\leq\infty$ be conjugate to $q$, that is, $\frac{1}{p}+\frac{1}{q}=1$.  Then $\mathcal{WD}_{\ell_p}^{(\infty,\xi)}(T_q^*)=\mathcal{L}(T_q^*)$, whereas $\mathcal{WD}_{\ell_p}^{(\infty,\omega_1)}(T_q^*)\neq\mathcal{L}(T_q^*)$.\end{example}

\begin{proof}Consider the identity operator $I:T_q^*\to T_q^*$.  We claim that $I\in\mathcal{WS}_{\ell_p}^{(2^{1/q},\xi)}(T_q^*)$.  Indeed, let $(x_n)$ be a normalized weakly null sequence in $T_q^*$, and let $\epsilon>0$.  By the uniform version of the Bessaga-Pe\l czy\'{n}ski Selection Principle, there exists a subsquence $(x_{n_k})$ which is $(1+2^{-1/q}\epsilon)$-equivalent to a normalized block basic sequence $(u_k)$ of the unit vector basis.  Thus, by Lemma \ref{Tsirelson-block-estimate}, for every $(\alpha_k)\in\mathcal{A}_\xi$ we have $\norm{\sum\alpha_kx_{n_k}}_{T_q^*}\leq(1+2^{-1/q}\epsilon)\norm{\sum\alpha_ku_k}_{T_q^*}\leq(2^{1/q}+\epsilon)\norm{(\alpha_k)}_{\ell_p}$, and the claim is proved.

On the other hand, we also claim $I\notin\mathcal{WD}_{\ell_p}^{(\infty,\omega_1)}(T_q^*)$.  Let $(e_n)$ be the unit vector basis of $T_q^*$, which is also weakly null since $T_q^*$ is reflexive.  Recall from Lemma \ref{convexified-lower-estimate} that $\norm{(\alpha_n)}_{\ell_p}\leq\norm{(\alpha_n)}_{T_q^*}$ for all $(\alpha_n)\in c_{00}$.  Hence, for any subsequence $(n_k)$ we have $\norm{\sum\alpha_ke_{n_k}}_{T_q^*}\geq\norm{(\alpha_k)}_{\ell_p}$.  Since $T_q^*$ fails to contain a copy of $\ell_p$, then for any $C\geq 0$ and $\epsilon>0$ we must now be able to find some $(\alpha_k)\in c_{00}$ with $\norm{\sum\alpha_ke_{n_k}}_{T_q^*}\geq(C+\epsilon)\norm{(\alpha_k)}_{\ell_p}$.\end{proof}

\bibliographystyle{amsplain}

\end{document}